\documentclass[11pt]{amsart}

\usepackage{amssymb}	
\usepackage{tikz}
\usepackage{tikz-cd}

\usepackage{amsmath}
 
\usepackage{mathrsfs}
\usepackage{amsthm}
\usepackage{verbatim}
\usepackage{subeqnarray}
 
\usepackage{color}
\usepackage{float}
\usepackage{bm}
\usepackage{hyperref}
\usepackage{amsfonts}
\usepackage{amscd}
\usepackage{cases}
\usepackage{xcolor}

\definecolor{red}{rgb}{0.8,0,0}
\definecolor{darkorange}{rgb}{1,0.4,0}
\definecolor{lightorange}{rgb}{1,0.6, 0}
\definecolor{yellow}{rgb}{1,0.8, 0}

\newtheorem{theorem}{Theorem}
\newtheorem{remark}{Remark}

\newtheorem{lemma}{Lemma}

\newcommand\tr{\operatorname{tr}}
\newcommand\inc{\operatorname{inc}}
\newcommand\skw{\operatorname{skw}}
\newcommand\sskw{\operatorname{sskw}}
\newcommand\vskw{\operatorname{vskw}}
\newcommand\mskw{\operatorname{mskw}}

\newcommand\sym{\operatorname{sym}}
\newcommand\grad{\operatorname{grad}}
\newcommand\deff{\operatorname{def}}
\renewcommand\div{\operatorname{div}}
\renewcommand\ker{\mathcal{N}}

\newcommand\curl{\operatorname{curl}}
\newcommand\rot{\operatorname{rot}}

\newcommand\ein{\operatorname{ein}}

\newcommand\hess{\operatorname{hess}}

\newcommand\ran{\mathcal{R}}

\newcommand\K{\mathbb{K}}
\newcommand\M{\mathbb{M}}

\renewcommand\S{{\mathbb S}}

\newcommand\R{\mathbb{R}}

\newcommand\x{\times}

\newcommand\V{{\mathbb{V}}}

\usepackage{chngpage}

\usepackage{lipsum}

\usepackage{booktabs}

\newcommand{\bs}{{\scriptscriptstyle \bullet}}

\makeatletter
\@addtoreset{equation}{section}
\makeatother

\usepackage{geometry}
\begin{document}
\title{Nonlinear elasticity complex and a finite element diagram chase}

\begin{abstract}
In this paper, we present a nonlinear version of the linear elasticity (Calabi, Kr\"oner, Riemannian deformation) complex which encodes isometric embedding, metric, curvature and the Bianchi identity. We reformulate the rigidity theorem and a fundamental theorem of Riemannian geometry as the exactness of this complex. Then we generalize an algebraic approach for constructing finite elements for the Bernstein-Gelfand-Gelfand (BGG) complexes. In particular, we discuss the reduction of degrees of freedom with injective connecting maps in the BGG diagrams. We derive a strain complex in two space dimensions with a diagram chase. 
\end{abstract}

\author{Kaibo Hu}
 \address{Mathematical Institute,
University of Oxford, Andrew Wiles Building, Radcliffe Observatory Quarter, 
Oxford, OX2 6GG, UK}
\email{Kaibo.Hu@maths.ox.ac.uk}
\date{\today. Manuscript prepared for proceedings of the INdAM conference ``Approximation Theory and Numerical Analysis meet Algebra, Geometry, Topology'', which was held in September 2022 at Cortona, Italy.}
\maketitle

\section{Introduction}

A complex refers to a sequence of spaces $V^{k}, k=0, \cdots, n$ and operators $d^{k}: V^{k}\to V^{k+1}, k=0, \cdots, n-1$, such that $d^{k+1}\circ d^{k}=0$ for any $k$. In this paper, we focus on differential complexes, where $d^{\bs}$ are differential operators. Differential complexes play an important role in geometry, analysis and numerical partial differential equations \cite{arnold2018finite,Arnold.D;Falk.R;Winther.R.2006a,Arnold.D;Falk.R;Winther.R.2010a,arnold2021complexes,bru1992hilbert,vcap2001bernstein}. Finite element exterior calculus (FEEC) \cite{arnold2018finite,Arnold.D;Falk.R;Winther.R.2006a,Arnold.D;Falk.R;Winther.R.2010a} was inspired by finite element methods for electromagnetic and elasticity problems. A key idea of FEEC is to discretize problems with spaces that fit in complexes. For electromagneitc problems, the de~Rham complex encodes the key structures of the continuous problem; while for elasticity problems,  the elasticity (Calabi, Kr\"oner, Riemannian deformation) complex \cite{Arnold2006a,eastwood2000complex,eastwood1999variations,pauly2020elasticity} plays a similar role. 

Most existing literature focuses on complexes with linear differential operators. With linear operators, cohomology can be defined and carries rich information (for example, finite dimensional cohomology further implies various analytic results \cite{arnold2021complexes}). In some cases, e.g., the elasticity complex, the operators are obtained as a linearization of nonlinear versions. The space and nonlinear operators also form a complex.   Nevertheless, these original complexes with nonlinear operators (referred to as {\it nonlinear complexes}   below) have received less attention. 

It is of interest to investigate the nonlinear complexes, as they
\begin{enumerate}
\item  incorporate important structures for nonlinear models and PDEs;
\item provide a motivation for the linear versions;
\item  may play a role in clarifying some problems in numerical analysis of nonlinear problems (e.g., the stability/instability of different choices of finite elements for nonlinear elasticity \cite{angoshtari2017compatible,auricchio2013approximation}).
\end{enumerate}
Moreover, concepts are more clear in the context of nonlinear elasticity, such as pullbacks and frame indifference. 
In the first part of this paper, we present the nonlinear version of the elasticity complex and show that the exactness of this complex corresponds to rigidity and a fundamental theorem of Riemannian geometry \cite{ciarlet2013linear} (thus clarifying (1) and (2) above).
  
In the second part of the paper, we generalize an approach for constructing finite element versions of the linear elasticity complex (the strain and stress complexes in two space dimensions). The elasticity complex is a special case of the Bernstein-Gelfand-Gelfand (BGG) sequences \cite{arnold2021complexes,vcap2022bgg,vcap2001bernstein,eastwood1999variations}. The BGG machinery establishes a link between de~Rham complexes and more complicated ones, e.g., the elasticity complex \cite{arnold2021complexes,vcap2022bgg}. Specifically, one connects several copies of de~Rham complexes by algebraic connecting maps and eliminates some components from the diagram.
 On the discrete level, one may mimic the construction and fit finite element spaces in diagrams. We will refer to this approach of deriving more finite elements from de~Rham complexes with various kinds of regularity as a {\it finite element diagram chase}.
 The main difficulty of this approach is to construct finite element de~Rham complexes that fulfil the requirement, as the input complexes should have different kinds of regularity. This can be seen in the simplest case, where one connects two 1D de~Rham complexes and get a BGG complex with a second order operator (see  Figure \ref{fig:1Ddiagram} and Figure \ref{fig:1DBGG}).   In higher dimensions, the diagram chase requires an input of more delicate finite element de~Rham complexes. 
\begin{figure}[H]
\begin{center}
\includegraphics[width=1.5in]{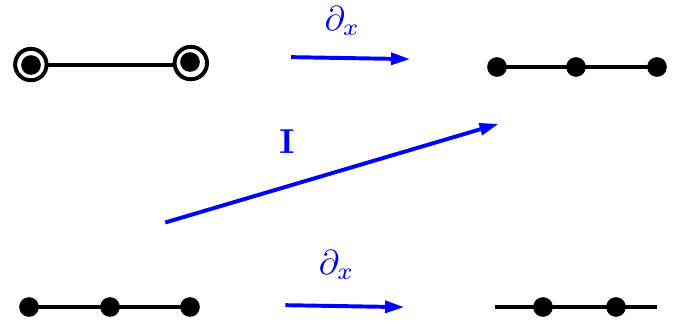} 
\end{center}
\caption{BGG diagram for deriving a complex with a second order operator. The two rows are finite element de~Rham complexes with different kinds of regularity.}\label{fig:1Ddiagram}
 \end{figure}
 \begin{figure}[H]
\begin{center}
\includegraphics[width=1.5in]{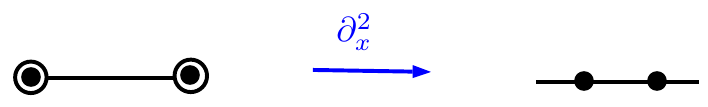}
\end{center}
\caption{The derived BGG complex with a second order operator.}
\label{fig:1DBGG}
 \end{figure}

 


A finite element diagram chase was first used in \cite{Arnold2006a} as a re-interpretation of the Arnold-Winther elasticity element. Recently, this approach has been extended to other complexes and finite elements
\cite{christiansen2020discrete,christiansen2018nodal,christiansen2022finite}.  In applications related to the Hellinger-Reissner principle (e.g., \cite{Arnold2006a}), the connecting maps are surjective. In this case, one may define a discrete version of the connecting map by rewriting the degrees of freedom \cite{Arnold2006a}. In other examples, the connecting maps can be injective. The reduction of spaces and degrees of freedom is more subtle. A finite element elasticity complex on the Alfeld split in 3D was constructed in \cite{christiansen2020discrete}. The first connecting map $\mskw$ is injective.  
A discrete BGG diagram plays an important role in \cite{christiansen2020discrete}. Nevertheless, the strain space was derived directly, rather than following a systematic diagram chase. 

In this paper, we demonstrate the reduction of spaces and degrees of freedom by closely following the BGG diagrams. We use the 2D strain complex as an example. The derived space is not symmetric but is so in a discrete sense (see Section \ref{sec:diagram-chase} for more detailed explanations). We also discuss a stress complex on quadrilateral grids. 

Discretizing the entire BGG diagrams is not merely a way to derive discrete BGG sequences, but it may have the benefit of incorporating more physics. With each BGG diagram, we can gather spaces from all rows into a {\it twisted complex} where the operators have an off-diagonal part mixing the components from different rows \cite{arnold2021complexes,vcap2022bgg}.  The twisted complexes encode more information than the BGG sequences, as the latter can be obtained by eliminating some spaces from the former (and the main conclusion of the BGG machinery is that this elimination is cohomology-preserving). This elimination has a physical meaning. For example, for the elasticity complex, the Hodge-Laplacian problem of the twisted complex corresponds to a linear Cosserat model which involves both displacement and a pointwise rotation. Then to derive the elasticity complex, one eliminates the rotational degree of freedom and obtains standard elasticity with only displacement as the main unknown. In \cite{vcap2022bgg}, it was observed that such a correspondence exists in 2D for the Reissner-Mindlin plate and the Kirchhoff-Love plate, and in 3D for the Cosserat and standard elasticity. Here we further observe that the Hodge-Laplacian problems for the 1D twisted and BGG complexes (continuous versions of Figure \ref{fig:1Ddiagram} and Figure \ref{fig:1DBGG}), with the energy functionals $\|\partial_{x}u-w\|_{A}^{2}+\|\partial_{x}w\|_{B}^{2}$ and $\|\partial_{x}^{2}w\|^{2}_{C}$,  correspond to the linearized Timoshenko beam and Euler-Bernoulli beam, respectively.  Here $u$ is the displacement; $w$ is the deflection; $\|\cdot\|_{\bs}$ are weighted norms. Therefore deriving a discrete version of the entire diagram may also shed a light on these problems. 

Differential complexes also provide a connection between scalar splines and tensor-valued spaces with relatively lower regularity. In particular, exactness of a complex $(V^{\bs}, d^{\bs})$ gives an equality of dimensions $\sum_{k}(-1)^{k}\dim V^{k}=0$. This connection also inspires the construction of tensor-valued finite elements. Specifically, one may also start with a smooth scalar spline space in a complex, differentiate it and obtain tensor-valued spaces that complete the sequence. Table \ref{tab:connections} summarizes some scalar and tensor-valued spaces and complexes that connect them. 


The rest of the paper is organized as the following.  In Section \ref{sec:nonlinear-complex}, we discuss a complex with nonlinear differential operators. In Section \ref{sec:diagram-chase}, we derive  finite element strain and stress complexes following a diagram chase. In Section \ref{sec:conclusion}, we provide concluding remarks.
\begin{table}[H]
\begin{tabular}{c|c|c} 
{2D Stokes} (2013) \cite{falk2013stokes}&Argyris (1968) \cite{argyris1968tuba} & Falk-Neilan (2013) \cite{falk2013stokes}\\
{2D Stokes } (2018) \cite{christiansen2018generalized}&Clough-Tocher (1965) \cite{clough1965finite} & Arnold-Qin (1992) \cite{arnold1992quadratic}\\
{2D Stokes} (2020) \cite{guzman2020exact}&Powell-Sabin (1977) \cite{zhang2008p1} & Zhang (2008) \cite{zhang2008p1}\\
{2D elasticity } (2002)  \cite{arnold2002mixed} &Argyris (1968) \cite{argyris1968tuba} &   Arnold-Winther (2002) \cite{arnold2002mixed} 
\\
{2D elasticity } (2018) \cite{christiansen2018nodal} & Argyris (1968) \cite{argyris1968tuba} &Hu-Zhang (2014) \cite{hu2014family}
\\
{2D elasticity } (2022) \cite{christiansen2022finite}&Clough-Tocher (1965) \cite{clough1965finite} & Johnson-Mercier (1978) \cite{johnson1978some}\\
{3D Stokes } (2015) \cite{neilan2015discrete}& Z\v{e}n\'{\i}\v{s}ek (1973) \cite{vzenivsek1973polynomial} & Neilan (2015) \cite{neilan2015discrete} \\
{3D Stokes } (2020) \cite{fu2020exact}&Alfeld (1984) \cite{alfeld1984trivariate} & Zhang (2005) \cite{zhang2005new}\\
{3D Stokes } (2022)\cite{guzman2022exact}& Worsey-Farin (1987) \cite{worsey1987ann} & Zhang (2011) \cite{zhang2011quadratic}\\
{3D elasticity } (2020) \cite{christiansen2020discrete}& Alfeld (1984) \cite{alfeld1984trivariate} &   \begin{tabular}{@{}c@{}}Christiansen-Gopalakrishnan\\ \quad \quad-Guzman-Hu (2020) \cite{christiansen2020discrete}\end{tabular}
    \end{tabular}
    \caption{Complexes connect scalar/vector splines with tensor-valued spaces. The second column of the table includes scalar/vector spline spaces, and the third contains the corresponding tensor-valued spaces. The first column indicates the complexes that contain these spaces. See also \cite{neilan2020stokes} for a review of finite element Stokes complexes and \cite{hu2022oberwolfach} for some BGG complexes. }
    \label{tab:connections}
    \end{table}

\section{Nonlinear complex}\label{sec:nonlinear-complex}

In this section, we present a nonlinear version of the linear elasticity complex.

Equip $\mathbb{R}^{3}$ with a flat metric $g_{0}$. 
Let $\Omega\subset \mathbb{R}^{3}$ be a bounded domain, and $\varphi_{0}: \Omega\to \varphi_{0}(\Omega)\subset \mathbb{R}^{3}$ be a given invertible smooth vector-valued function.   We refer to the following sequence as a nonlinear elasticity complex:
\begin{equation}\label{nonlinear-complex}
\begin{tikzcd}
\mathbb{RM}\arrow{r}{\subset}&C^{\infty}({\Omega}; \mathbb{R}^{3})\arrow{r}{\iota} &\Gamma_{+}(S^{2}T^{\ast}\Omega)   \arrow{r}{\mathrm{Ric}}&\Gamma(S^{2}T^{\ast}\Omega) \arrow{r}{\mathrm{Bian}}&T^{\ast}\Omega \arrow{r}{}&0.
 \end{tikzcd}
\end{equation}
Here $\mathbb{RM}$ is the space of rigid body motions composed with $\varphi_{0}$, i.e., 
$$
\mathbb{RM}:=\{c+Q\cdot \varphi_{0}(x), ~c\in \mathbb{R}^{3}, Q\in SO(3)\}. 
$$
The space $C^{\infty}({\Omega}; \mathbb{R}^{3})$ has the geometric meaning of embedding of $\Omega$ into $\mathbb{R}^{3}$. Each such map $\varphi$ induces a pullback of metric $\varphi^{\ast}: \Gamma(S^{2}T^{\ast}\mathbb{R}^{3})\to \Gamma(S^{2}T^{\ast}\Omega)$, where $\Gamma(S^{2}T^{\ast}M)$ denotes the section of symmetric $(0, 2)$ tensor fields on a manifold $M$. The operator $\iota$ denotes the change of the metric, i.e., $\iota (\varphi):=(\varphi)^{\ast}g_{0}-(\varphi_{0})^{\ast}g_{0}$, and in coordinate forms, $\iota(\varphi):=D\varphi\cdot g_{0}\cdot (D\varphi)^{t}-D\varphi_{0}\cdot g_{0}\cdot (D\varphi_{0})^{t}$, where $(D\varphi)_{ij}:=\partial_{i}\varphi_{j}$. For any $r\in \mathbb{RM}$, $r$ induces the same map on the metric as $\varphi_{0}$, as $r$ is a composition of $\varphi_{0}$ with a rigid body motion. Then it is readily seen that $\iota (r)=0$ for any $r\in \mathbb{RM}$.  The space $\Gamma_{+}(S^{2}T^{\ast}\Omega) $ is defined by
$$
\Gamma_{+}(S^{2}T^{\ast}\Omega):=\{g\in \Gamma(S^{2}T^{\ast}\Omega): g+(\varphi_{0})^{\ast}g_{0}>0\}.
$$
Here a matrix $g+(\varphi_{0})^{\ast}g_{0}>0$ means that $g+(\varphi_{0})^{\ast}g_{0}$ is positive definite.  In \eqref{nonlinear-complex}, $\mathrm{Ric}$ is the operator mapping a metric to the Ricci curvature. In three dimensions, the Riemannian, Ricci and Einstein tensors encode the same information. Moreover, $\mathrm{Bian}$ is the operator in the (contracted) differential Bianchi identity $R_{ij}\mapsto \nabla^{i}(R_{ij}-\frac{1}{2}Rg_{ij})$, where $R_{ij}$ is the Ricci tensor, $R:=g^{ij}R_{ij}$ is the scalar curvature, and $g_{ij}$ is the metric tensor. Note that $R_{ij}-\frac{1}{2}Rg_{ij}$ is the Einstein tensor. 

In continuum mechanics, $\varphi$ represents a map from the reference configuration $\Omega$ to the current configuration $\varphi(\Omega)$. When $g_{0}=I$ is the Euclidean metric of $\mathbb{R}^{3}$, $\varphi^{\ast}(g_{0})$, or in the matrix form, $D\varphi\cdot (D\varphi)^{t}$, is the right Cauchy-Green tensor \cite[Definition 3.5]{marsden1994mathematical}.

\begin{theorem}
The sequence \eqref{nonlinear-complex} is a complex in the sense that composing two successive operators gives zero.
\end{theorem}
\begin{proof}
We have $\mathrm{Ric}\circ \iota=0$ as $\varphi_{0}$ is a flat metric by assumption, and $\operatorname{Bian}\circ\operatorname{Ric}=0$ is the Bianchi identity. 
\end{proof}

 \begin{remark}
 Another complex related to nonlinear elasticity was investigated in \cite{angoshtari2016hilbert}.  The complex includes the embedding $\varphi$ and the deformation tensor $D \varphi$, and therefore corresponds to a de~Rham complex. The operators in the complex in \cite{angoshtari2016hilbert} are linear. Thus structures of multiplication are not explicitly included in the sequences. This is different from \eqref{nonlinear-complex}, which contains $D \varphi\cdot (D \varphi)^{t}$.
 \end{remark}

In the rest of this section, we discuss exactness of the nonlinear complex and show how it corresponds to some classical results. For simplicity, hereafter we assume that $g_{0}$ is the Euclidean metric, or in coordinates, $g_{0}=I$.

The exactness at index zero (at the space $C^{\infty}({\Omega}; \mathbb{R}^{3})$) states that if $D\varphi\cdot (D\varphi)^{T}-D\varphi_{0}\cdot (D\varphi_{0})^{T}=0$, then $\varphi$ is in $\mathbb{RM}$, i.e., there exists $c, Q$ such that $\varphi=c+Q\varphi_{0}$. This is exactly the rigidity theorem \cite[Theorem 8.7-1]{ciarlet2013linear} which is stated for $C^{1}$ maps and connected open subsets of $\mathbb{R}^{n}$. The connectedness is a natural condition, as in the linear case exactness relies on trivial topology.

The exactness at index one states that for any $g\in \Gamma_{+}(S^{2}T^{\ast}\Omega) $, if the Ricci (Riemann, Einstein) tensor vanishes, then there exists $\varphi_{0}$ such that $g=D\varphi\cdot (D\varphi)^{t}-D\varphi_{0}\cdot (D\varphi_{0})^{t}$. This is equivalent to applying the theorem on ``the existence of an immersion with a prescribed metric tensor'' \cite[Theorem 8.6-1]{ciarlet2013linear} to $g+D\varphi_{0}\cdot (D\varphi_{0})^{t}$. In fact, $g\in \Gamma_{+}(S^{2}T^{\ast}\Omega)$ implies that $g+D\varphi_{0}\cdot (D\varphi_{0})^{t}$ is positive definite. Moreover, $g$ is flat implies that   $g+D\varphi_{0}\cdot (D\varphi_{0})^{T}$ also has vanishing Riemnnian tensor. Then we have all the conditions for applying \cite[Theorem 8.6-1]{ciarlet2013linear}, which asserts that the exactness holds  if $\Omega$ is a simply connected open set in $\mathbb{R}^{n}$ and $g\in C^{2}$ (therefore the embedding $\phi$ in $C^{3}$). This again corresponds to the topological conditions for the exactness of the linear elasticity complex. 
 
As shown above, the exactness of \eqref{nonlinear-complex} at index zero corresponds to the rigidity and at index one is a special case of the fundamental theorem of Riemannian geometry. It is tempting to investigate the exactness at index two. Generally speaking, the question is that assuming that a tensor satisfies the symmetries of a Ricci (Riemannian) tensor, is it a curvature tensor of some metric? However, this is not a well-posed question, as to formulate the Bianchi identity, one needs a metric given beforehand. This issue was discussed in \cite{rendall1989insufficiency}, where the author investigated another possibility of imposing the compatibility condition as ``there {\it exists} some metric $g$, such that $\operatorname{Bian}_{g}\sigma =0$'', where $\operatorname{Bian}_{g}$ is the Bianchi operator defined by $g$.  Nevertheless, with this condition, $\sigma$ may not be the Riemannian tensor of any metric, i.e., the ``exactness'' does not hold.


The linearization of \eqref{nonlinear-complex} around the Euclidean metric gives the linear elasticity complex
 \begin{equation}\label{sequence:3Delasticity}
\begin{tikzcd}
\mathrm{RM} \arrow{r}&  C^{\infty}\otimes \mathbb{V} \arrow{r}{{\deff}} &C^{\infty}\otimes  \mathbb{S}  \arrow{r}{\inc} &   C^{\infty}\otimes \mathbb{S} \arrow{r}{\div} & C^{\infty}\otimes \mathbb{V}\arrow{r}&0,
 \end{tikzcd}
\end{equation}
where we write $\mathbb{V}=\mathbb{R}^{3}$ for the space of vectors, and $C^{\infty}\otimes\mathbb{V}=C^{\infty}(\Omega)\otimes \mathbb{V}$ is the space of smooth vector fields.  Similarly, $\mathbb{S}$ and $C^{\infty}\otimes \mathbb{S}$ denote the spaces of $3\x3$ symmetric matrices and smooth matrix fields, respectively. Here $\mathrm{RM}$ is the space of  infinitesimal rigid body motions
$$
\mathrm{RM}:=\{c+K\cdot x, ~c\in \mathbb{R}^{3}, Q\in {SO}(3)\}=\{c+b\times x, ~c, b\in \mathbb{R}^{3}\}. 
$$
The operators in the elasticity complex are the  {deformation} or linearized strain operator $\operatorname{def}=\operatorname{sym}\grad$, the  {incompatibility} operator $\inc =\curl\circ \mathrm{T}\circ\curl$ (where $\mathrm{T}$ denotes the transpose operation and $\curl$ acts on a matrix field by columns), and the column-wise {divergence} operator map matrices to vectors.

\section{Finite element diagram chase}\label{sec:diagram-chase}

In this section, we present a diagram chase to construct a finite element version of elasticity complexes in 2D.

Following \cite{arnold2021complexes}, we introduce some notations in 2D.  We use $\R$, $\mathbb{V}$, $\mathbb{M}$, $\mathbb{S}$, and $\mathbb{K}$ to denote the (finite dimensional) spaces of scalars, vectors, matrices, symmetric matrices and skew-symmetric matrices, respectively. That is, $\mathbb{V}:=\mathbb{R}^{2}$, $\mathbb{M}:=\mathbb{R}^{2\times 2}$, and $\mathbb{S}:=\mathbb{R}^{2\times 2}_{\sym}$. The operators $\skw: \M\to \K$ and $\sym:\M\to\S$ are the skew-symmtrization and symmetrization, respectively.
In $\mathbb{R}^{2}$, a skew symmetric matrix can be identified with a scalar. Let $\mskw: \mathbb{R}\to \mathbb{K}$ be this identification, i.e.,  
$$
\mskw(u):= \left ( 
\begin{array}{cc}
0 &u\\
-u & 0
\end{array}
\right )\quad \mbox{in } \mathbb{R}^{2}.
$$
We also let $\sskw=\mskw^{-1}\circ \skw: \mathbb{M}\to \mathbb{R}$ be the map taking the skew part of a matrix and identifying it with a scalar. 

We use $H^{q}$ to denote the Sobolev space for which all derivatives of order less than or equal to $q$ are square integrable. 
For a linear operator $D$, define $H(D):=\{u\in L^{2}: Du\in L^{2}\}$.

\subsection{Stress and strain complexes: an overview}
 In 2D, there are two complexes related to \eqref{sequence:3Delasticity}. Related to the Hellinger-Reissner principle, one has the {\it stress complex}
\begin{equation}\label{2D-stress}
\begin{tikzcd}[column sep=3em]
0 \arrow{r} &H^{2} \arrow{r}{\curl\curl} &H(\div, \S)\arrow{r}{\div}&L^{2}\otimes \mathbb{V}  \arrow{r}{} & 0,
 \end{tikzcd}
\end{equation}
where $H(\div, \S):=\{\sigma\in L^{2}\otimes \mathbb{S}: \div \sigma\in L^{2}\otimes \mathbb{V}\}$. The stress complex may be viewed as a restriction of the last part of the complex \eqref{sequence:3Delasticity} to a two dimensional face (thus $\inc$ becomes $\curl\curl$ and $\div$ remains). The stress complex can be derived from the following diagram:
\begin{equation} \label{stress-diagram}
\begin{tikzcd}[column sep=1.6em]
0 \arrow{r} &H^{2} \arrow{r}{\curl} &H^{1}\otimes \mathbb{V}\arrow{r}{\div}& L^{2}  \arrow{r}{} & 0 \\
0 \arrow{r}&H^{1}\otimes \mathbb{V} \arrow{r}{\curl} \arrow[ur, "I"] &H(\div)\otimes \mathbb{V}\arrow{r}{\div} \arrow[ur, "\mathrm{-2\sskw}"]& L^{2}\otimes \V\arrow{r}{} &0.
 \end{tikzcd}
\end{equation}
Applications for the Hellinger-Reissner principle of elasticity use the last part of the diagram. On the continuous level, the idea of deriving \eqref{2D-stress} from the diagram \eqref{stress-diagram} is to eliminate spaces connected by the algebraic diagonal maps as much as possible from the diagram. For example, we remove the two spaces connected by $I$ and the components connected by $\sskw$ (thus $L^{2}$ is removed from the top row and the symmetric part of $H(\div)\otimes \mathbb{V}$ is left in the bottom row). We refer to \cite{arnold2021complexes,vcap2022bgg} for a precise discussion of the BGG machinery.  

To derive finite elements, we intend to fit finite element spaces in the diagram \eqref{stress-diagram}:
\begin{equation} \label{stress-diagram-h}
\begin{tikzcd}[column sep=1.6em]
0 \arrow{r} &W^{0} \arrow{r}{\curl} &W^{1}\arrow{r}{\div}&W^{2}  \arrow{r}{} & 0 \\
0 \arrow{r}&Y^{0} \arrow{r}{\curl} \arrow[ur, "I"] &Y^{1}\arrow{r}{\div} \arrow[ur, "\mathrm{-2\sskw}"]& Y^{2}\arrow{r}{} &0.
 \end{tikzcd}
\end{equation}
 One first constructs a finite element Stokes complex $W^{\bs}$ (de~Rham complexes with enhanced regularity) for the first row. Then one chooses $Y^{0}$ to be identical to $W^{1}$, so that they are matched precisely by the connecting map $I$.  

  Once we have a diagram with discrete spaces \eqref{stress-diagram-h}, a similar argument as the continuous level leads to a discretization of \eqref{2D-stress}. The spaces in the derived complex are kernels and cokernels of the connecting algebraic maps. The recipe does not immediately implies the degrees of freedom. Nevertheless, one may rewrite the degrees of freedom for $Y^{1}$, identify some of them with the degrees of freedom for $W^{2}$ and set them to zero. The remaining ones give a set of degrees of freedom for symmetric tensors. We refer to \cite{Arnold2006a,christiansen2018nodal} for details of this argument.

Another version of elasticity complexes in 2D, the strain complex 
\begin{equation}\label{strain-complex}
\begin{tikzcd}[column sep=3em]
0 \arrow{r} &H^{2}\otimes \mathbb{V}\arrow{r}{\deff} &H^{1}(\rot, \S)\arrow{r}{\rot\rot}&L^{2}  \arrow{r}{} & 0,
 \end{tikzcd}
\end{equation}
is a restriction of the first several spaces of the 3D complex \eqref{sequence:3Delasticity} to a 2D plane (with higher Sobolev regularity, c.f. \cite[(64)]{christiansen2022finite}).  Here $H^{1}(\rot, \S):=\{u\in H^{1}\otimes \mathbb{S}: \rot u\in H^{1}\otimes \mathbb{V}\}$. The strain complex \eqref{strain-complex} contains displacement and linearized Gauss curvature (see appendix). The complex \eqref{strain-complex} can be derived from the following diagram 
\begin{equation}\label{Z-complex-2D}
\begin{tikzcd}[column sep=1.6em]
0 \arrow{r} &H^{2}\otimes \mathbb{V}\arrow{r}{\grad} &H^{1}(\rot, \M)\arrow{r}{\rot}&H^{1}\otimes \V  \arrow{r}{} & 0 \\
0 \arrow{r}&H^{2} \arrow{r}{\grad} \arrow[ur, "\mskw"] &H^{1}\otimes \V\arrow{r}{\rot} \arrow[ur, "I"]& L^{2}\otimes \V\arrow{r}{} &0,
 \end{tikzcd}
\end{equation}
where   $H^{1}(\rot, \mathbb{M}):=\{e\in H^{1}\otimes \mathbb{M}: \rot e\in H^{1}\otimes \mathbb{V}\}$. On the discrete level, to fit finite element spaces   in a discrete version of \eqref{Z-complex-2D},
\begin{equation} \label{discrete-strain-diagram}
\begin{tikzcd}[column sep=1.6em]
0 \arrow{r} &Z^{0} \arrow{r}{\grad} &Z^{1}\arrow{r}{\rot}& Z^{2}  \arrow{r}{} & 0 \\
0 \arrow{r}&V^{0} \arrow{r}{\grad} \arrow[ur, "\mskw"] &V^{1}\arrow{r}{\rot} \arrow[ur, "I"]&V^{2}\arrow{r}{} &0.
 \end{tikzcd}
\end{equation}
 we may follow the same idea as for the stress complex \eqref{stress-diagram} by using identical spaces for the two slots connected by $I$, i.e., $V^{1}=Z^{2}$.
 However, the elimination of $V^{0}$ from $Z^{1}$ is more subtle than the stress complex. To see this, we first observe that in   \eqref{stress-diagram-h}, a diagram chase shows that as long as $\sskw$ maps $Y^{1}$ to $W^{2}$ (which is often the case as $W^{2}$ requires minimal regularity)  and both rows are exact, then $\sskw$ is also surjective on the discrete level, i.e., $W^{2}$ precisely corresponds to the skew-symmetric part of $Y^{1}$. 
 Contrarily, 
 the subtlety for \eqref{Z-complex-2D} is that  there is no 
canonical way to choose the two spaces connected by $\mskw$ such that $\mskw V^{0}$ is precisely the skew-symmetric part of $Z^{1}$. Actually, in typical examples, $\mskw$ maps $Z^{0}$ to $V^{1}$, but $V^{0}$ is often too small compared to the skew-symmetric part of $Z^{1}$. 

In fact, this difference between \eqref{stress-diagram-h} and \eqref{discrete-strain-diagram} is already seen on the continuous level in the framework of \cite{arnold2021complexes}. In \cite{arnold2021complexes}, a BGG complex is derived from two input complexes. There exists some $J>0$, such that the connecting maps $S^{j}$ are injective for $j\geq J$ and surjective for $j\geq J$ (therefore bijective for $j=J$) \cite[(19)]{arnold2021complexes}.  The derivation for the two parts follows slightly different forms. For example, the derived operators \cite[(26)]{arnold2021complexes} contain projections to the cokernel for small $j$ (e.g., $\sym\grad=P_{\ran(\mskw)^{\perp}}\grad$), while for large $j$ the operators directly follow from the input complexes (e.g., $\div$).
This partly explains the difference and difficulty in dealing with \eqref{strain-complex} and \eqref{stress-diagram-h} compared to existing results for \eqref{2D-stress} \cite{Arnold2006a,christiansen2018nodal}. 


\subsection{A discrete stress complex on quadrilateral meshes}

We derive a discrete stress complex from a BGG diagram. The idea is similar to existing results on other types of meshes \cite{Arnold2006a,christiansen2018nodal,christiansen2022finite}. Therefore we will only give a brief outline. Nevertheless, the resulting complex seems to be new in the literature. 

In the construction, we split a quadrilateral by connecting the opposite vertices. This leads to four sub-triangles. All the finite element spaces are defined on this subdivision (referred to as the criss-cross grid; see Figure \ref{fig:JM}).

\begin{center}
\begin{figure}[ht]
\setlength{\unitlength}{1cm}

\begin{picture}(2,4)(3,0)

\put(-1.9, 2.4){$\mathbb{R}$}
\put(-1.6, 2.5){\vector(1, 0){0.4}}

\put(0,2){
\put(-1, 0){\circle*{0.1}} 
\put(-0.5, 1){\circle*{0.1}} 
\put(1, 0){\circle*{0.1}} 
\put(1, 1){\circle*{0.1}} 
\put(-1, 0){\circle{0.2}} 
\put(-0.5, 1){\circle{0.2}} 
\put(1, 0){\circle{0.2}} 
\put(1, 1){\circle{0.2}} 
\put(0.25, 1){\vector(0, 1){0.3}}
\put(1, 0.5){\vector(1, 0){0.3}}
\put(-0.75, 0.5){\vector(-1, 0){0.3}}
\put(0, 0){\vector(0, -1){0.3}}

\put(-1, 0){\line(1,2){0.5}} 
\put(-1,0){\line(1,0){2}}
\put(1,0){\line(0,1){1}}
\put(-0.5,1){\line(1,0){1.5}}
\put(-1,0){\line(2, 1){2}} 
\put(-0.5,1){\line(3, -2){1.5}} 
}

\put(1.5, 2.5){\vector(1, 0){1}}
\put(1.68, 2.65){$\curl$}

\put(4,2){
\put(-0.95, 0){\circle*{0.1}} 
\put(-1.05, 0){\circle*{0.1}} 
\put(-0.45, 1){\circle*{0.1}} 
\put(-0.55, 1){\circle*{0.1}} 
\put(1.05, 0){\circle*{0.1}} 
\put(0.95, 0){\circle*{0.1}} 
\put(1.05, 1){\circle*{0.1}} 
\put(0.95, 1){\circle*{0.1}} 

\put(-0.1, 0.4){+10} 

\put(0.2, 1){\circle*{0.1}}
\put(0.3, 1){\circle*{0.1}}
\put(1, 0.45){\circle*{0.1}}
\put(1, 0.55){\circle*{0.1}}
\put(-0.77, 0.45){\circle*{0.1}}
\put(-0.73, 0.55){\circle*{0.1}}
\put(-0.05, 0){\circle*{0.1}}
\put(0.05, 0){\circle*{0.1}}

\put(-1, 0){\line(1,2){0.5}} 
\put(-1,0){\line(1,0){2}}
\put(1,0){\line(0,1){1}}
\put(-0.5,1){\line(1,0){1.5}}
\put(-1,0){\line(2, 1){2}} 
\put(-0.5,1){\line(3, -2){1.5}} 
}

\put(5.5, 2.5){\vector(1, 0){1}}
\put(5.75, 2.65){{$\div$}}

\put(8,2){
\put(-0.1, 0.4){+11} 

\put(-1, 0){\line(1,2){0.5}} 
\put(-1,0){\line(1,0){2}}
\put(1,0){\line(0,1){1}}
\put(-0.5,1){\line(1,0){1.5}}
\put(-1,0){\line(2, 1){2}} 
\put(-0.5,1){\line(3, -2){1.5}}
}

\put(1.5, 1.2){\vector(2, 1){1}}
\put(1.7, 1.5){$I$}

\put(-2, 0.4){$\mathbb{R}^{2}$}
\put(-1.6, 0.5){\vector(1, 0){0.4}}
\put(0,0){

\put(-0.95, 0){\circle*{0.1}} 
\put(-1.05, 0){\circle*{0.1}} 
\put(-0.45, 1){\circle*{0.1}} 
\put(-0.55, 1){\circle*{0.1}} 
\put(1.05, 0){\circle*{0.1}} 
\put(0.95, 0){\circle*{0.1}} 
\put(1.05, 1){\circle*{0.1}} 
\put(0.95, 1){\circle*{0.1}} 

\put(-0.1, 0.4){+10} 

\put(0.2, 1){\circle*{0.1}}
\put(0.3, 1){\circle*{0.1}}
\put(1, 0.45){\circle*{0.1}}
\put(1, 0.55){\circle*{0.1}}
\put(-0.77, 0.45){\circle*{0.1}}
\put(-0.73, 0.55){\circle*{0.1}}
\put(-0.05, 0){\circle*{0.1}}
\put(0.05, 0){\circle*{0.1}}

\put(-1, 0){\line(1,2){0.5}} 
\put(-1,0){\line(1,0){2}}
\put(1,0){\line(0,1){1}}
\put(-0.5,1){\line(1,0){1.5}}
\put(-1,0){\line(2, 1){2}} 
\put(-0.5,1){\line(3, -2){1.5}} 

}

\put(1.5, 0.5){\vector(1, 0){1}}
\put(1.68, 0.65){$\curl$}

\put(5.5, 1.2){\vector(2, 1){1}}
\put(5.7, 1.5){$\mathrm{sskw}$}

\put(4,0){

\put(0, 1){\vector(0, 1){0.3}}
\put(0.05, 1){\vector(0, 1){0.3}}
\put(0.55, 1){\vector(0, 1){0.3}}
\put(0.6, 1){\vector(0, 1){0.3}}

\put(1, 0.75){\vector(1, 0){0.3}}
\put(1, 0.8){\vector(1, 0){0.3}}
\put(1, 0.25){\vector(1, 0){0.3}}
\put(1, 0.2){\vector(1, 0){0.3}}

\put(-0.9, 0.2){\vector(-1, 0){0.3}}
\put(-0.85, 0.25){\vector(-1, 0){0.3}}
\put(-0.7, 0.65){\vector(-1, 0){0.3}}
\put(-0.65, 0.7){\vector(-1, 0){0.3}}

\put(-0.2, 0){\vector(0, -1){0.3}}
\put(-0.15, 0){\vector(0, -1){0.3}}
\put(0.3, 0){\vector(0, -1){0.3}}
\put(0.35, 0){\vector(0, -1){0.3}}
\put(-0.1, 0.4){+16}

\put(-1, 0){\line(1,2){0.5}} 
\put(-1,0){\line(1,0){2}}
\put(1,0){\line(0,1){1}}
\put(-0.5,1){\line(1,0){1.5}}
\put(-1,0){\line(2, 1){2}} 
\put(-0.5,1){\line(3, -2){1.5}} 
}

\put(5.5, 0.5){\vector(1, 0){1}}
\put(5.75, 0.65){{$\div$}}

\put(8,0){
\put(-0.1, 0.4){+8} 

\put(-1, 0){\line(1,2){0.5}} 
\put(-1,0){\line(1,0){2}}
\put(1,0){\line(0,1){1}}
\put(-0.5,1){\line(1,0){1.5}}
\put(-1,0){\line(2, 1){2}} 
\put(-0.5,1){\line(3, -2){1.5}}
}

\end{picture}
\end{figure}
\end{center}

\begin{center}
\begin{figure}[H]
\setlength{\unitlength}{1cm}
\begin{picture}(2,1)(3,0)
\put(0,0){
\put(-1, 0){\circle*{0.1}} 
\put(-0.5, 1){\circle*{0.1}} 
\put(1, 0){\circle*{0.1}} 
\put(1, 1){\circle*{0.1}} 
\put(-1, 0){\circle{0.2}} 
\put(-0.5, 1){\circle{0.2}} 
\put(1, 0){\circle{0.2}} 
\put(1, 1){\circle{0.2}} 
\put(0.25, 1){\vector(0, 1){0.3}}
\put(1, 0.5){\vector(1, 0){0.3}}
\put(-0.75, 0.5){\vector(-1, 0){0.3}}
\put(0, 0){\vector(0, -1){0.3}}

\put(-1, 0){\line(1,2){0.5}} 
\put(-1,0){\line(1,0){2}}
\put(1,0){\line(0,1){1}}
\put(-0.5,1){\line(1,0){1.5}}
\put(-1,0){\line(2, 1){2}} 
\put(-0.5,1){\line(3, -2){1.5}} 
}

\put(1.5, 0.5){\vector(1, 0){1}}
\put(1.3, 0.65){$\curl\curl$}

\put(4,0){

\put(0, 1){\vector(0, 1){0.3}}
\put(0.05, 1){\vector(0, 1){0.3}}
\put(0.55, 1){\vector(0, 1){0.3}}
\put(0.6, 1){\vector(0, 1){0.3}}

\put(1, 0.75){\vector(1, 0){0.3}}
\put(1, 0.8){\vector(1, 0){0.3}}
\put(1, 0.25){\vector(1, 0){0.3}}
\put(1, 0.2){\vector(1, 0){0.3}}

\put(-0.9, 0.2){\vector(-1, 0){0.3}}
\put(-0.85, 0.25){\vector(-1, 0){0.3}}
\put(-0.7, 0.65){\vector(-1, 0){0.3}}
\put(-0.65, 0.7){\vector(-1, 0){0.3}}

\put(-0.2, 0){\vector(0, -1){0.3}}
\put(-0.15, 0){\vector(0, -1){0.3}}
\put(0.3, 0){\vector(0, -1){0.3}}
\put(0.35, 0){\vector(0, -1){0.3}}
\put(-0.1, 0.4){+5}

\put(-1, 0){\line(1,2){0.5}} 
\put(-1,0){\line(1,0){2}}
\put(1,0){\line(0,1){1}}
\put(-0.5,1){\line(1,0){1.5}}
\put(-1,0){\line(2, 1){2}} 
\put(-0.5,1){\line(3, -2){1.5}} 
}

\put(5.5, 0.5){\vector(1, 0){1}}
\put(5.75, 0.65){{$\div$}}

\put(8,0){
\begin{picture}(2,2)
\put(-0.1, 0.4){+8} 

\put(-1, 0){\line(1,2){0.5}} 
\put(-1,0){\line(1,0){2}}
\put(1,0){\line(0,1){1}}
\put(-0.5,1){\line(1,0){1.5}}
\put(-1,0){\line(2, 1){2}} 
\put(-0.5,1){\line(3, -2){1.5}} \end{picture}
}

\end{picture}
\caption{Diagram chase for a quadrilateral macroelement. We denote the top row by $(W^{\bs}, d^{\bs})$ and the bottom row by $(Y^{\bs}, d^{\bs})$, following the notation in \eqref{stress-diagram-h}.}
\label{fig:JM}
\end{figure}
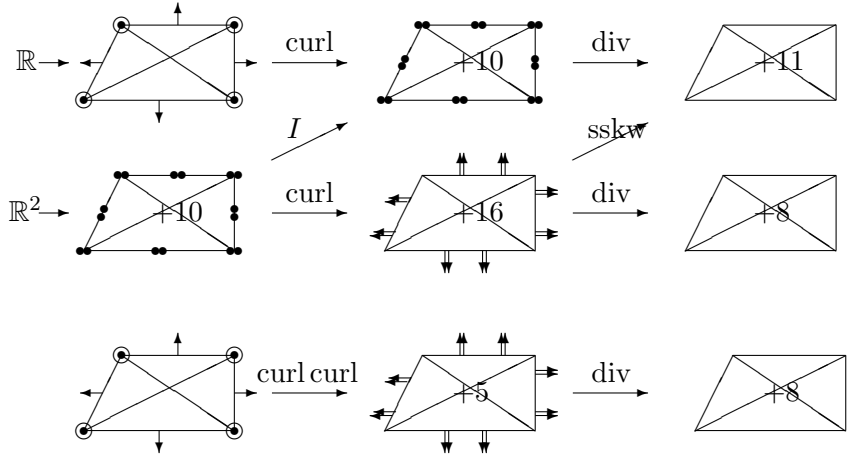
\end{center}

We will use the notation in \eqref{stress-diagram-h} to define the spaces.  The sequence in the top row is studied in \cite{criss-cross}. 
The scalar spline space $W^{0}$ can be found in \cite{lai2007spline}. The space consists of piecewise cubic polynomials with continuous derivatives.  The last two spaces in the first row were discussed by Arnold and Qin \cite{arnold1992quadratic} as a Stokes pair. Specifically, $W^{1}$ consists of continuous piecewise quadratic polynomials. Therefore there are 10 interior degrees of freedom (2 at each of the 5 Lagrange points inside the refined quadrilateral). It is clear that $\div$ maps $W^{1}$ to piecewise linear polynomials, but this map is not surjective. In fact,  for any $u\in W^{1}$, $\div u$ satisfies a relation at the intersection of the two diagonal lines. Figure \ref{fig:z-condition} shows a special configuration where the two lines are the $x$- and $y$-axises, respectively. Then one verifies that 
$$
(\div u|_{T_{1}}-\div u|_{T_{2}}+\div u|_{T_{3}}-\div u|_{T_{4}})(z)=0.
$$
The situation for general configurations where the two lines are not orthogonal is similar.
\begin{figure}[H]
\begin{center}
\includegraphics[width=0.2\linewidth]{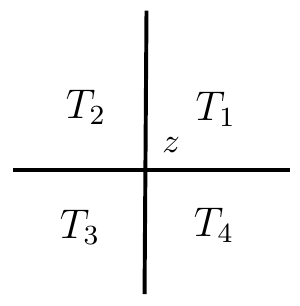}
\end{center}
\caption{A compatibility condition is imposed at each singular vertex $z$. }
\label{fig:z-condition}
\end{figure}
With this constraint, piecewise linear polynomials have dimension 11 on each quadrilateral. 
Vertices joined by two straight lines are referred to as singular vertices \cite{arnold1992quadratic}.
A similar situation appears in the construction of finite elements on the Powell-Sabin split, where two straight lines intersect on edges of the triangulation \cite{guzman2020exact}. 

The second row $(Y^{\bs}, d^{\bs})$ in \eqref{stress-diagram-h} is a vector-valued (two copies of) finite element de~Rham complex consisting of quadratic Lagrange elements, piecewise linear Brezzi-Douglas-Marini (BDM) elements \cite{brezzi1985two}, and piecewise constants on the criss-cross grids. 

Similar to the continuous level, to get the BGG complex, we start from $W^{0}$, and connect the two rows to obtain a second order operator $\curl\curl$. Then we restrict to $\ker(\sskw)\subset Y^{1}$. The degrees of freedom of $\ker(\sskw)$ can be obtained by eliminating the degrees of freedom of $W^{2}$ from $Y^{1}$. More precisely,  all the degrees of freedom on the edges are retained, while we get $5=16-11$ interior degrees of freedom. The finite element stress complex reads
\begin{equation}\label{discrete-stress-comp}
\begin{tikzcd}
0 \arrow{r} & \mathcal{P}_{1} \arrow{r}{\subset} & W^{0} \arrow{r}{\curl\curl} & \ker(\sskw)  \arrow{r}{\div}  &Y^{2} \arrow{r}{} &0.
\end{tikzcd}
\end{equation}
Standard argument with a dimension count implies the following. 
\begin{theorem}
The sequence \eqref{discrete-stress-comp} is a complex. On contractible domains, \eqref{discrete-stress-comp} is exact except for index zero, where $\ker(\curl\curl)=\mathcal{P}_{1}$, the space of linear polynomials. 
\end{theorem}

 Johnson and Mercier \cite{johnson1978some} constructed a stable finite element elasticity pair on the criss-cross grid on quadrilaterals. In \cite{johnson1978some}, each component of the displacement (corresponding to $Y^{2}$) consists of linear polynomials on quadrilaterals (rather than the refined triangles). Therefore the local shape function space on each quadrilateral has dimension 6. The stress space has the same edge degrees of freedom as $\ker(\sskw)$ in \eqref{discrete-stress-comp}, but has 3 interior degrees of freedom, instead of 5 for $\ker(\sskw)$. The pair in \cite{johnson1978some}  thus does not fit in a complex as derivatives of piecewise functions on a quadrilateral are not a polynomial in general.

\subsection{A new finite element strain complex}

In this section, we construct finite elements fitting in \eqref{discrete-strain-diagram}, and derive the resulting discrete BGG complex for \eqref{strain-complex}. In particular, we discuss a discrete reduction of the skew-symmetric components of the strain tensor. The finite elements for   \eqref{discrete-strain-diagram} are summarized in Figure \ref{fig:strain-diagram}.
\begin{figure}[H]
\begin{center}
\includegraphics[width=0.7\linewidth]{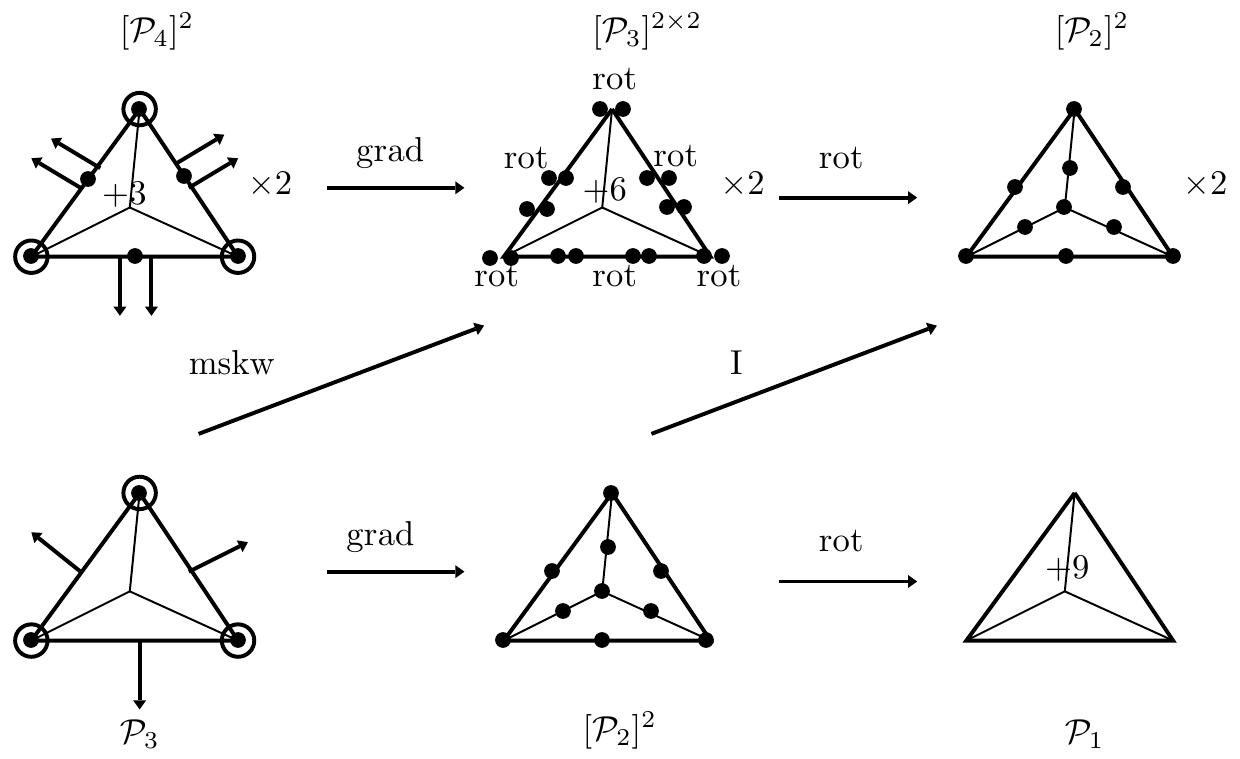}
\end{center}
\caption{Finite element  diagram \eqref{discrete-strain-diagram} for deriving a discrete strain complex. }
\label{fig:strain-diagram}
\end{figure}

All these finite elements will be defined on a Clough-Tocher split of a triangle, which refers to a split where we choose an interior point in the triangle and connect it to the three vertices. Therefore a triangle is split into three smaller ones (c.f. \cite{lai2007spline}). 
 A {Clough-Tocher mesh} consists of triangles with Clough-Tocher splits.

In the discussions below, we will use $\mathcal{P}_{r}(S)$ to denote the space of polynomials of degree less than or equal to $r$ on a domain $S$. Typically, $S$ is a triangle or its edges. We use $\mathcal{T}_{h}$ to denote a Clough-Tocher mesh, and $T\in \mathcal{T}_{h}$ is a macroelement consisting of three triangles. We use $T_{i}\in T_{\mathrm{CT}}, i=1, 2, 3,$ to denote the three sub-triangles of $T$. Moreover,  $\Delta_{k}(\mathcal{T}_{h})$, or simply $\Delta_{k}$, is the set of all $k$-dimensional subsimplicies of $\mathcal{T}_{h}$. If $e\in \Delta_{1}$ is an edge of $\mathcal{T}_{h}$, then $\tau_{e}$ and $n_{e}$ are the unit tangent and normal vectors, respectively. We use $\Delta_{k}(T)$ to denote the set of $k$-simplices in $T$.

The complex $(Z^{\bs}, d^{\bs})$ on the top row of \eqref{discrete-strain-diagram} consists of a Clough-Tocher $H^{2}$ element, an $H^{1}(\rot)$ element, and a Lagrange element on the Clough-Tocher mesh.  The lowest order version of this complex (starting with piecewise cubic polynomials) can be found in \cite[Proposition 3]{christiansen2018generalized}.

{\noindent --} $Z^{0}$. Note that $Z^{0}$ is   vector-valued. Each component of  $Z^{0}$ is the Clough-Tocher $C^{1}$ element with piecewise quartic polynomials. The lowest order of such elements with piecewise cubic polynomials can be found in, e.g., \cite{lai2007spline}, while high order versions on a triangle can be found in \cite{fu2020exact}. The local shape functions on a macroelement $T$ are 
$$
\mathcal{P}_{Z_{0}}(T):=\{u_{h}\in C^{1}(T): u_{h}|_{T_{i}}\in \mathcal{P}_{4}(T_{i})\otimes \mathbb{V},\, T_{i}\subset T_{\mathrm{CT}}, ~ i=1, 2, 3\}.
$$

The degrees of freedom for $w\in Z^{0}$ can be given by 
\begin{itemize}
\item   $w(x), \partial_{l} w(x), \quad x\in \Delta_{0}(T), ~l=1, 2$,
\item  $\int_{e}\partial_{n}wq\,ds, \quad q\in \mathcal{P}_{1}(e), ~~e\in \Delta_{1}(T)$,
\item 3 interior degrees of freedom.
\end{itemize}
Here $\partial_{n}w:=n\cdot \grad w$.

{\noindent --} $Z^{1}$. The space $Z^{1}$ consists of two copies of  an $H^{1}(\rot)$-conforming element (continuous fields with continuous $\rot$). The local shape function space for a single copy (each column of the matrix) is 
$$
\mathcal{P}_{Z_{1}}(T):=\{u_{h}\in C^{0}(T):  \rot u_{h}\in C^{0}(T),  u_{h}|_{T_{i}}\in \mathcal{P}_{3}(T_{i})\otimes \mathbb{M},\, T_{i}\subset T_{\mathrm{CT}}, ~ i=1, 2, 3\}.
$$
 The degrees of freedom for $\mathcal{P}_{Z_{1}}(T)$ can be given by the following:
\begin{subequations}
\label{dof-Z1}
\begin{alignat}{4}
\label{dof-Z1-1}
u(a), \quad\forall a\in \Delta_{0}(T), \\
\label{dof-Z1-1-1}
\rot u(a), \quad\forall a\in \Delta_{0}(T), \\
\label{dof-Z1-2}
 \int_{e}uq, \quad\forall q\in \mathcal{P}_{1}(e),\quad\forall e\in \Delta_{1}(T),\\
\label{dof-Z1-3}
\int_{e}\rot u\cdot \tau_{e}, \quad \forall e\in \Delta_{1}(T),\\
\label{dof-Z1-3-1}
\int_{e}\rot u\cdot n_{e}, \quad \forall e\in \Delta_{1}(T),\\
\label{dof-Z1-4}
\int_{T} u: q, \quad \forall q\in \mathcal{C}(T),  \, T\in \Delta_{2}(T),
\end{alignat}
\end{subequations}
where $\mathcal{C}(T)$ denotes the bubble space on $T$. 

The concrete form of the interior degrees of freedom is less important in the sense that they do not affect the interelement continuity. The dimension of $Z^{i}, i=0, 1, 2$ and the bubble spaces (the dimension of interior degrees of freedom) can be found in \cite[Section 3]{fu2020exact}.

{\noindent --} $Z^{2}$. The space $Z^{2}$ consists of two copies of the scalar quadratic Lagrange elements on the Clough-Tocher split. 

\begin{theorem}
The degrees of freedom in \eqref{dof-Z1} are unisolvent, and the complex $(Z^{\bs}, d^{\bs})$ consisting of global finite element spaces is exact on contractible domains (except for index 0).
\end{theorem}
\begin{proof}
The unisolvency of $Z^{i}$ is a direct consequence of the dimension formulas in \cite[Section 3]{fu2020exact}. In fact, we can show that if all the vertex and edge degrees of freedom vanish,  then the field consists of bubbles in the sense of \cite{fu2020exact}.  Then vanishing integral against bubbles implies that the field actually vanishes.

To show the exactness, we first observe that for any $u\in Z^{1}$, $u$ is also in the N\'ed\'elec space. Then $\rot u=0$ implies that $u=\grad\phi$ for some $\phi$ in the Lagrange space (exactness of standard finite element de~Rham sequence). As $u$ is $C^{0}$, we conclude that $\phi$ is $C^{1}$, and therefore $\phi\in Z^{0}$. The exactness then  follows from a dimension count $1=\dim Z^{0}-\dim Z^{1}+\dim Z^{2}$ with the Euler number $1=V-E+T$, where $V$, $E$ and $T$ are the numbers of vertices, edges and triangles, respectively. 
\end{proof}
The cohomology of the bottom row $(V^{\bs}, d^{\bs})$ can be found in \cite{christiansen2018generalized}.

 The bottom complex in \eqref{discrete-strain-diagram} is the Clough-Tocher finite element Stokes complex \cite[(59)]{christiansen2018generalized}. The first space is a scalar Clough-Tocher $C^{1}$ element, the second is a vector-valued quadratic Lagrange element, and the last consists of piecewise linear polynomials. 
 
 In the rest of this section, we start from \eqref{discrete-strain-diagram} and derive a discrete version of the strain complex 
\begin{equation}\label{reduced-complex}
\begin{tikzcd}[column sep=3em]
0 \arrow{r} &Z^{0}\arrow{r}{\deff_{h}} &U^{1}\arrow{r}{\rot\rot}&V^{2}  \arrow{r}{} & 0.
 \end{tikzcd}
\end{equation}
The first and the last spaces in \eqref{reduced-complex} are directly read from \eqref{discrete-strain-diagram}. The space $U^{1}$ is $Z^{1}$ with reduced symmetry (roughly speaking, obtained by eliminating the $V^{0}$ component from $Z^{1}$).  The key difference compared to the reduction of skew-symmetric components in existing constructions for the stress complex \cite{Arnold2006a,christiansen2018nodal} is that $U^{1}$ is symmetric only in a discrete sense, and $\deff_{h}$ is different from the $\deff$ operator on the continuous level which leads to symmetric fields.  See \eqref{def:U1} and \eqref{def:defh} below for precise definitions.

  We first observe the following identity.
 \begin{lemma}\label{lem:identities}
 We have
$ \rot \skw u=-\grad\sskw u$, and therefore $\rot u=-\grad\sskw u+\rot \sym u$.
 \end{lemma}
Note that $\sskw(Z^{1})$ is not in $V^{0}$. Inspired by the identities in Lemma \ref{lem:identities}, we define a discrete $\sskw$ operator, $\sskw_{h}: Z^{1}\to V^{0}$, by setting the degrees of freedom of $V^{0}$:
\begin{subequations}
\label{cond}
\begin{alignat}{4}
\label{cond-1}
\sskw_{h}u(a):=\sskw u(a), \quad \forall a\in \Delta_{0},\\
\label{cond-2}
\grad\sskw_{h}u(a):=-\rot u(a), \quad \forall a\in \Delta_{0},\\
\label{cond-3}
\int_{e} \partial_{n}\sskw_{h}u:=-\int_{e}n\cdot \rot u, \quad \forall e\in \Delta_{1}.
\end{alignat}
\end{subequations}
We have $\sskw_{h}\mskw=I$, since
\begin{subequations}
\begin{alignat}{4}
\sskw_{h}\mskw v(a)=\sskw \mskw v(a)=v(a), \quad \forall a\in \Delta_{0},\\
\grad\sskw_{h}\mskw v(a)=-\rot \mskw v(a)=\grad v(a), \quad \forall a\in \Delta_{0},\\
\int_{e} \partial_{n}\sskw_{h}\mskw v=-\int_{e}n\cdot \rot \mskw v=\int_{e}\partial_{n}  v, \quad \forall e\in \Delta_{1},
\end{alignat}
\end{subequations}
i.e., for any $v\in V^{0}$, the degrees of freedom have the same value on $\sskw_{h}\mskw v$ and $v$.

Define the shape function space of $U^{1}$ as the kernel of $\sskw_{h}$ (in $Z^{1}$)
\begin{equation}\label{def:U1}
U^{1}:=\ker(\sskw_{h}, Z^{1}),
\end{equation}
 and 
 \begin{equation}\label{def:defh}
 \deff_{h}:=(I-\mskw\sskw_{h})\grad.
 \end{equation}
  Then we get a sequence
\begin{equation}\label{reduced-complex}
\begin{tikzcd}[column sep=3em]
0 \arrow{r} &Z^{0}\arrow{r}{\deff_{h}} &U^{1}\arrow{r}{\rot\rot}&V^{2}  \arrow{r}{} & 0.
 \end{tikzcd}
\end{equation}
We note that $\deff_{h}Z^{0}\subseteq U^{1}$ since $\sskw_{h}\deff_{h}=\sskw_{h}(I-\mskw\sskw_{h})\grad=0$. Moreover, $\rot\rot\deff_{h}=\rot\rot(I-\mskw\sskw_{h})\grad= 0$, since $\rot\rot\mskw=0$ by a diagram chase on \eqref{discrete-strain-diagram}. So \eqref{reduced-complex} is a complex.

Next, we show the exactness of \eqref{reduced-complex} (except for index zero). First, we observe that $\rot\rot: Z^{1}\to V^{2}$ is surjective, as $\rot: Z^{1}\to Z^{2}$ and $\rot: V^{1}\to V^{2}$ are both so, and $Z^{2}=V^{1}$. The following theorem states the exactness at $U^{1}$. 
\begin{theorem}
$\ker(\rot\rot, U^{1})=\deff_{h}Z^{0}$.
\end{theorem}
\begin{proof}
From the exactness of the following complex  (see \cite[Proposition 2.3]{christiansen2020discrete})
\begin{equation}
\begin{tikzcd}[column sep=1.6em]
Z^{0}\arrow{r}{\grad} &Z^{1}\arrow{r}{\rot\rot}&V^{2}  \arrow{r}{} & 0 \\
V^{0} \arrow[ur, "\mskw"]
 \end{tikzcd}
\end{equation}
 we have
$$
\ker(\rot\rot, Z^{1})=\grad Z^{0}+\mskw V^{0}.
$$
This implies that for any $u\in U^{1}$, $\rot\rot u=0$, there exists $\phi\in Z^{0}$, $\psi\in V^{0}$, such that 
\begin{equation}\label{u}
u=\grad\phi+\mskw \psi.
\end{equation} 
Therefore
$$
0=\sskw_{h} u=\sskw_{h}\grad\phi+\sskw_{h}\mskw \psi=\sskw_{h}\grad\phi+ \psi.
$$
Consequently, $\psi=-\sskw_{h}\grad\phi$. Substituting this identity into \eqref{u},
$$
u=\grad\phi-\mskw\sskw_{h}\grad\phi=(I-\mskw\sskw_{h})\grad\phi=\deff_{h}\phi.
$$
\end{proof}
Next, we define the degrees of freedom of $U^{1}$ by eliminating degrees of freedom from $Z^{1}$ that correspond to the skew-symmetric components (eliminating $V^{0}$ from $Z^{1}$). 

The degrees of freedom of $U^{1}$ can be given as follows:
\begin{subequations}
\label{dof-U1}
\begin{alignat}{4}
\label{dof-U1-1}
\sym u(a), \quad\forall a\in \Delta_{0}(T), \\
\label{dof-U1-2}
 \int_{e}uq, \quad\forall q\in \mathcal{P}_{1}(e),\quad\forall e\in \Delta_{1}(T),\\
\label{dof-U1-3}
\int_{e}\rot u\cdot \tau_{e}, \quad \forall e\in \Delta_{1}(T),\\
\label{dof-U1-4}
\int_{T} u: q, \quad \forall q\in \mathcal{C}(T),  \, T\in \Delta_{2}(T). 
\end{alignat}
\end{subequations}

\begin{theorem}
The degrees of freedom in \eqref{dof-U1} are unisolvent for $U^{1}$.
\end{theorem}
\begin{proof}
First we show that if all the degrees of freedom in \eqref{dof-U1} vanish on $u\in U^{1}$, then $u=0$.  To see this, we will use \eqref{dof-U1} and the fact that $u\in \ker(\sskw_{h})$ to verify that 
 all the degrees of freedom in  \eqref{dof-Z1} (i.e., functionals in \eqref{cond}) vanish on $u$. Then from the unisolvence of  \eqref{dof-Z1} on $Z^{1}$, we obtain $u=0$ since    $U^{1}\subset Z^{1}$. To this end, we verify the following facts:
\begin{itemize}
\item $u(a)=0, \quad\forall a\in \Delta_{0}(T)$,\quad by \eqref{dof-U1-1} and \eqref{cond-1}
\item $\rot u(a)=-\grad\sskw_{h}u(a)=0, \quad\forall a\in \Delta_{0}(T)$, \quad\mbox{by \eqref{cond-2}}
\item $\int_{e}uq=0,\quad \forall q\in \mathcal{P}_{1}(e),\quad\forall e\in \Delta_{1}(T)$, \quad by \eqref{dof-U1-2}
\item $\int_{e}\rot u\cdot \tau_{e}=0, \quad \forall e\in \Delta_{1}(T)$, \quad by \eqref{dof-U1-3}
\item $\int_{e}\rot u\cdot n_{e}=-\int_{e} \partial_{n}\sskw_{h}u=0, \quad \forall e\in \Delta_{1}(T)$, \quad \mbox{by \eqref{cond-3}}
\item interior degrees of freedom vanish by definition. \quad by \eqref{dof-U1-4}
\end{itemize}
By the unisolvence of $Z^{1}$, we get $u=0$.

Then it remains to show that the number of degrees of freedom is equal to the dimension of the local shape function spaces. To this end, we note that $\mskw: V^{0}\to Z^{1}$ in \eqref{discrete-strain-diagram} is injective. Therefore the dimension of the local shape function space of $U^{1}$ satisfies $\dim U^{1}=\dim Z^{1}-\dim V^{0}$. On the other hand, we observe that the number of reduced degrees of freedom of \eqref{dof-U1} from \eqref{dof-Z1} is also equal to $\dim V^{0}$. Therefore $\dim U^{1}$ is equal to the number of the degrees of freedom in \eqref{dof-U1}.

\end{proof}

\begin{remark}
From the above construction, $U^{1}$ has $H^{1}(\rot)$-conformity since it is a subspace of $Z^{1}$. However,  the shape functions of $U^{1}$ may not be symmetric. 

\end{remark}

\begin{remark}[strong and weak symmetry]
In the context of the Hellinger-Reissner principle for linear elasticity, strong symmetry refers to finite elements which contain symmetric matrix-valued functions as local shape functions. As the construct of such finite elements is difficult in some circumstances, another approach was proposed where one uses a Lagrange multiplier to impose the symmetry conditions \cite{arnold2007mixed}. This is referred to as the approach of weak symmetry. The BGG diagram  plays a key role in choosing inf-sup stable finite element pairs for weak symmetry \cite{arnold2007mixed}. In this case, the two rows in \eqref{stress-diagram-h} do not match precisely in the sense that $Y^{0}$ is not identical to $W^{1}$. Instead, one chooses $Y^{0}$ to be larger than $W^{1}$ and uses an interpolation to replace $I$ in  \eqref{stress-diagram-h}. The weak symmetry approach is related to the strong symmetry in several ways. First, if the finite elements in a discrete BGG diagram match precisely (e.g., $\sskw Y^{1}=W^{2}$), then implementing the formulation of elasticity with weakly imposed symmetry leads to stress tensors with strong symmetry. This was discovered in \cite{gopalakrishnan2012second} referred to as ``serendipitous exact symmetry''. Second, in the case of ``serendipitous exact symmetry'', on the finite element level, one may further eliminate shape functions and degrees of freedom which correspond to the skew-symmetric components to obtain finite elements with strong symmetry  \cite{arnold2002mixed,christiansen2018nodal}. 

The finite element diagram chase in this paper fits in this picture, as we eliminate shape functions and degrees of freedom on the finite element level. Nevertheless, the strain complex calls for new techniques as the connecting map is injective. We cannot repeat the argument in \cite{arnold2002mixed,christiansen2018nodal} or  \cite{arnold2007mixed} which requires surjectivity. Extending formulations with weakly imposed symmetry to the part of the BGG diagrams where the connecting maps are injective remains open.
\end{remark}

\section{Conclusions}\label{sec:conclusion}

In this paper, we presented a nonlinear version of the elasticity complex. The spaces and operators encode key structures in elasticity theory and Riemannian geometry. The exactness at index zero and one corresponds to a rigidity theorem and a special case of equal-dimensional embedding of Riemannian manifolds. The Korn inequality also has nonlinear versions \cite{ciarlet2006nonlinear,ciarlet2015nonlinear}. Such results were established in the literature with various regularity assumptions, c.f., \cite{ciarlet2013linear,mardare2004isometric}. For the linearized versions, the BGG machinery provides a systematic approach for establishing the corresponding results.   Therefore, as a future direction, it is of interest to investigate insights that the BGG machinery and the linearized theory may bring to the nonlinear case.


In the second part of the paper, we constructed finite element elasticity complexes in 2D following a diagram chase. We focused on examples with the lowest order polynomials. The main purpose was to demonstrate the diagram chase with injective connecting maps.  The results may be generalized to higher order polynomials and higher dimensions (c.f., \cite{christiansen2020discrete}).

It is also of interest to investigate the role of nonlinear complexes in computation for nonlinear elasticity and geometric problems, e.g., \cite{angoshtari2017compatible,auricchio2013approximation}. It remains open to fit discrete (finite element) spaces in nonlinear complexes and preserve the exactness (rigidity and embedding theorems at the discrete level). A significant challenge is to clarify the algebraic varieties defined by the operators in the nonlinear complex.




\section*{Appendix: coordinate forms of linearized operators} 

In this appendix, we present the coordinate (vector/matrix) form of some linearized operators in geometry. In particular, this clarifies the terms in the elasticity complex. Similar formulas can be found in, e.g., \cite{christiansen2011linearization,li2018regge}.

Let  $g_{ij}$ be a Riemannian metric.  The Christoffel symbols have the form
$$
\Gamma_{jk}^{i}=\frac{1}{2}g^{il}\left ( \frac{\partial g_{lj}}{\partial x^{k}}+ \frac{\partial g_{lk}}{\partial x^{j}}-\frac{\partial g_{jk}}{\partial x^{l}}\right ),
$$
and their contracted versions
$$
\Gamma^{i}=\frac{1}{2}g^{jk}g^{il}\left (\frac{\partial g_{lj}}{\partial x^{k}}+\frac{\partial g_{lk}}{\partial x^{j}}-\frac{\partial g_{jk}}{\partial x^{l}} \right ).
$$

The linearization around the Euclidean metric is thus
\begin{align*}
\Gamma^{i}&\sim \frac{1}{2}\delta^{jk}\delta^{il}\left (\frac{\partial g_{lj}}{\partial x^{k}}+\frac{\partial g_{lk}}{\partial x^{j}}-\frac{\partial g_{jk}}{\partial x^{l}} \right )\\
&=\frac{\partial g_{ij}}{\partial x^{j}}-\frac{1}{2}\frac{\partial g_{j}^{j}}{\partial x^{l}}=\div g-\frac{1}{2}\nabla \tr (g)=\div S^{-1}g.
\end{align*}
This shows that the linearized contracted Christoffel symbol corresponds to the operator $\div S^{-1}$.

The linearized Ricci tensor around the Euclidean metric:
$$
R_{ij}=1/2(-\partial_{k}\partial^{k}g_{ij}-\partial_{i}\partial_{j}g_{k}^{k}+\partial_{i}\partial^{k}g_{kj}+\partial_{j}\partial^{k}g_{ik}),
$$
and in terms of vector/matrix notation,
$$
Ric=1/2(-\Delta g-\nabla\nabla \tr (g)+2\deff\div g).
$$
Taking the trace,
$$
R=-\Delta \tr(g)+\div\div g=-\div\div Sg.
$$
This corresponds to the scalar (Hamiltonian) constrain in the linearized Einstein equations \cite[(5.7)]{li2018regge}.
The linearization of the Einstein tensor $G_{ij}:=R_{ij}-1/2Rg_{ij}$ becomes
$$
G_{ij}=1/2\left [-\partial_{k}\partial^{k}g_{ij}-\partial_{i}\partial_{j}g_{k}^{k}+\partial_{i}\partial^{k}g_{kj}+\partial_{j}\partial^{k}g_{ik}+(\partial_{k}\partial^{k}g_{l}^{l}-\partial^{k}\partial^{l}g_{kl})\eta_{ij}\right],
$$
where $\eta_{ij}$ is the Euclidean metric, 
and in terms of vector/matrix notation,
\begin{align}\label{ein-expansion}
G&=1/2[-\Delta g-\hess \tr (g)+2\deff\div g+(\Delta  \tr (g)-\div\div g)I].
\end{align}

In 2D, by a straightforward calculation, we have 
\begin{equation}\label{R}
R=-\div\div Sg=\partial_{1}^{2}g_{22}+\partial_{2}^{2}g_{11}-2\partial_{1}\partial_{2}g_{12}=\rot\rot g.
\end{equation}
In 3D, we have 
$$
G=\frac{1}{2}\inc g,
$$
and $$
Ric=S^{-1}G=1/2S^{-1}\inc g.
$$
From \eqref{ein-expansion}, we have
$$
\tr(G)=\frac{1}{2}[\Delta \tr(g)-\div\div g]. 
$$
We have used the space dimension $n=3$ in the above calculation. Then we have
\begin{equation}\label{ric-expansion}
Ric=S^{-1}G=1/2[-\Delta g-\hess \tr (g)+2\deff\div g+1/2(\Delta  \tr (g)-\div\div g)I]
\end{equation}

\begin{lemma}
In three space dimensions, $R_{ijkl}=\frac{1}{2}G_{ijkl}$, where $G_{ijkl}:=\epsilon_{ij}^{\quad s}\epsilon_{kl}^{ \quad t}(\inc g)_{st}$ is the fourth-order tensor version of $\inc g$ (obtained by identifying each of the two indices of $\inc g$ with a ``skew-symmetric matrix''), and $R_{ijkl}$ is the linearized Riemannian tensor of $g$. Moreover, the linearized Einstein tensor $Ein=1/2\inc g$.
\end{lemma}
\begin{proof}

\begin{equation*}
\begin{aligned}
G_{pqcd} &=\epsilon_{pq}^{\quad i}\epsilon_{cd}^{\quad j}\epsilon_{iab}\epsilon_{jst}\partial^{a}\partial^{s}g^{bt}\\
&=\left(\delta_{p a} \delta_{qb}-\delta_{pb} \delta_{qa}\right)\left(\delta_{c s} \delta_{d t}-\delta_{ct} \delta_{d s}\right)\partial^{a}\partial^{s}g^{bt} \\
&=\left(\delta_{c s} \delta_{d t}-\delta_{ct} \delta_{d s}\right)\left(\partial_{p} \partial^{s} g_{q}^{t}-\partial_{q}  \partial^{s} g_{p}^{t}\right) \\
&=2\left(\delta_{c s} \delta_{d t}-\delta_{ct} \delta_{d s}\right)\left(\partial_{[p} \partial^{s} g_{q]}^{~~t}\right) \\
&=4\partial_{[p} \partial_{[[c} g_{q]d]]}, 
\end{aligned}
\end{equation*}
where $\partial_{[p} \partial_{[[c} g_{q]d]]}$ denotes the skew-symmetrization with $p$, $q$ and $c$, $d$, respectively.

On the other hand, we have the linearized Riemamian tensor $R_{ij~~~l}^{~~~\,\, k}=2 \partial_{[i} \Gamma_{j]l}^{k}$, and
\begin{align*}
\Gamma_{j l}^{k}&=\frac{1}{2} g^{ks}\left(\frac{\partial{g}_{sj}}{\partial{x}^{l}}+\frac{\partial{g}_{sl}}{\partial{x}^{j}}-\frac{\partial{g}_{jl}}{\partial{x}^{s}}\right)\sim \frac{1}{2}\left(\frac{\partial g^{k}_{~~j}}{\partial x^{l}}+\frac{\partial g^{k}_{~~l}}{\partial x^{j}}-\frac{\partial g_{jl}}{\partial x^{k}}\right) \sim g_{j}^{~~[k} \nabla_{l]}+\frac{1}{2} \frac{\partial g^{k}_{~~l}}{\partial x^{j}}.
\end{align*}
This implies that $R_{ij~~~l}^{~~~\,\, k}=2 \partial_{[i} \Gamma_{j]l}^{k}\sim 2 \partial_{[i} g_{j]}^{~~[[k}\nabla_{l]]}$, i.e., $Riem = 2\mskw\inc g$, where $\mskw$ maps each index of $\inc g$ to two indices corresponding to a skew-symmetric 2-tensor.

Taking trace on the identity $R_{pq, cd}=1/2\epsilon_{pq}^{~~~i}\epsilon^{~~~j}_{cd}\epsilon_{iab}\epsilon_{jst}\partial^{a}\partial^{s}g^{bt}$, we get
\begin{align*}
R_{qd}&=-\frac{1}{2}[\epsilon_{pq}^{ \quad i}\epsilon_{~~d}^{p~~j} \epsilon_{iab} \epsilon_{jst}\partial^{a}\partial^{s}g^{bt}]\\
&=-\frac{1}{2}[\left(\delta_{qd} \delta^{ij}-\delta_{q}^{j} \delta_{d}^{i}\right)  \epsilon_{iab} \epsilon_{jst}\partial^{a}\partial^{s}g^{bt}]\\
&= -\frac{1}{2}[ \epsilon_{iab} \epsilon^{i}_{~~st}\partial^{a}\partial^{s}g^{bt}\delta_{qd}-  \epsilon_{dab} \epsilon_{qst}\partial^{a}\partial^{s}g^{bt}]\\
&= -\frac{1}{2}[(\delta_{as}\delta_{bt}-\delta_{at}\delta_{bs})\partial^{a}\partial^{s}g^{bt}\delta_{qd}-  \epsilon_{dab} \epsilon_{qst}\partial^{a}\partial^{s}g^{bt}]\\
&=-\frac{1}{2}[ (\Delta \tr g-\div\div g)I-\inc g]_{qd}.
\end{align*}
We note that 
\begin{align*}
\tr \inc g&=\epsilon^{d}_{~~ab}\epsilon_{dst}\partial^{a}\partial^{s}g^{bt}\\
&=(\delta_{as}\delta_{bt}-\delta_{at}\delta_{bs})\partial^{a}\partial^{s}g^{bt}\\
&=\Delta\tr g-\div\div g.
\end{align*}
This implies that
$$
2Ric=S\inc g,
$$
and 
$$
Ein=S^{-1}Ric=\frac{1}{2}\inc g.
$$
\end{proof}

\section*{Acknowledgement}

The work was supported by a Hooke Research Fellowship and a Royal Society University Research Fellowship (URF$\backslash${\rm R}1$\backslash$221398). The author would like to thank Douglas~N.~Arnold for pointing out the reference  \cite{rendall1989insufficiency}. 
\bibliographystyle{siam}      
\bibliography{reference}{}   

\end{document}